\documentclass[11pt]{article}

\usepackage[latin1]{inputenc}
\usepackage{t1enc}
\usepackage{amsmath, amssymb, amsfonts, amsthm, amsopn,epsfig}
\usepackage[none]{hyphenat}
\usepackage{tikz}
\usepackage{float}
\usepackage{graphicx}
\usepackage{caption}
\usepackage[margin=1in]{geometry}
\usepackage[toc,page]{appendix}
\usepackage{epstopdf}
\usepackage{etoolbox}
\usepackage[colorlinks=true]{hyperref}
\usepackage{dsfont}
\hypersetup{
	colorlinks=true,
	linkcolor=blue,
	filecolor=magenta,      
	urlcolor=cyan,
	citecolor=blue
}
\usepackage{cleveref}

\theoremstyle{plain}
\newtheorem{theorem}{Theorem}[section]

\newtheorem{corollary}[theorem]{Corollary}
\newtheorem{claim}[theorem]{Claim}
\newtheorem{lemma}[theorem]{Lemma}

\theoremstyle{definition}
\newtheorem{definition}{Definition}

\title{Additive bases, coset covers, and non-vanishing linear maps}
\author{J\'anos Nagy\thanks{Alfr\'ed R\'enyi Institute of Mathematics and
MTA-BME Lend\"ulet Arithmetic Combinatorics Research Group, ELKH, 
\emph{email}: \textbf{janomo4@gmail.com}}
	\and
	P\'eter P\'al Pach\thanks{MTA-BME Lend\"ulet Arithmetic Combinatorics Research Group, ELKH, Department of Computer Science and Information Theory, Budapest University of Technology and Economics, \emph{email}: \textbf{ppp@cs.bme.hu} }
	\and
	Istv\'an Tomon\thanks{ETH Zurich, \emph{e-mail}: \textbf{istvan.tomon@math.ethz.ch}}}
\date{}

\begin{document}
\sloppy 
\maketitle

\begin{abstract}




Recently, the first two authors proved the Alon-Jaeger-Tarsi conjecture on non-vanishing linear maps, for large primes. We extend their ideas to address several other related conjectures. 

We prove the weak Additive Basis conjecture proposed by Szegedy, making a significant  step towards  the  Additive Basis conjecture of Jaeger, Linial, Payan, and Tarsi. In fact, we prove it in a strong form: there exists a set $A\subset\mathbb{F}_p^*$ of size $O(\log p)$ such that if $B\subset\mathbb{F}_p^{n}$ is the union of $p$ linear bases, then 
$A\cdot B=\{a\cdot v:a\in A, v\in B\}$ is an additive basis.

An old result of Tomkinson states that if $G$ is a group, and $\{H_{i}x_{i}:i\in [k]\}$ is an irredundant coset cover of $G$, then $|G:\bigcap_{i\in [k]} H_{i}|\leq k!,$ and this bound is the best possible. It is a longstanding open problem whether the upper bound can be improved to $e^{O(k)}$ in case we restrict cosets to subgroups. Pyber proposed to study this question for abelian groups. We show that somewhat surprisingly, if $G$ is abelian, the upper bound can be improved to $e^{O(k\log \log k)}$ already in the case of general coset covers, making the first substantial improvement over the $k!$ bound.

Finally, we prove a natural generalization of the Alon-Jaeger-Tarsi conjecture for multiple matrices.
\end{abstract}

\section{Introduction}

It was observed by Szegedy \cite{Sz07} that there is a group of longstanding conjectures in linear algebra and group theory that are closely related. This includes the Alon-Jaeger-Tarsi conjecture \cite{AT89,J82} on non-vanishing linear maps, the Additive basis conjecture of Alon, Linial, and Meshulam \cite{ALM91}, and problems of Neumann \cite{N54} and Pyber \cite{P96} on minimal coset covers of groups. The Alon-Jaeger-Tarsi conjecture states that if $p\geq 5$, then for every invertible matrix $M\in \mathbb{F}_p^{n\times n}$ there exists $x\in\mathbb{F}_p^{n}$ such that neither $x$, nor $Mx$ have zero coordinates. 

The first two authors investigated connections between the polynomial method and different versions of the group ring method in  \cite{NP21}. As a combinatorial shadow of these ideas they resolved the Alon-Jaeger-Tarsi conjecture for large $p$. The final combinatorial proof can be presented quite shortly. On the other hand it would have been difficult to come up with this proof without the other more conceptual ideas. In the current paper, as a continuation of this program we use the combinatorial ideas from \cite{NP21} to make substantial progress on each of the aforementioned problems.


\subsection{Additive Basis conjecture}
If $p$ is a prime, and $n$ is a positive integer, a multiset $B\subset \mathbb{F}_p^{n}$ is called an \emph{additive basis}, if every vector $w\in \mathbb{F}_p^{n}$ can be written as $w=\sum_{v\in B}\alpha_v v$, where $\alpha_{v}\in\{0,1\}$ for every $v\in B$. Clearly, if $B_0$ is a linear basis, and $B$ is the union of $p-1$ copies of $B_0$, then $B$ is an additive basis (here and later, union is taken as a multiset). The Additive Basis conjecture of Jaeger, Linial, Payan, and Tarsi \cite{JLPT} asks whether there exists a constant $c_1(p)$ (possibly $c_1(p)=p$) such that the union of $c_1(p)$ linear bases is always an additive basis. More precisely, this conjecture first appeared in a work of Alon, Linial, and Meshulam \cite{ALM91}, who proved that the union of $\Omega(p\log n)$ bases is an additive basis in $\mathbb{F}_p^{n}$, but in \cite{ALM91}, the conjecture is attributed to~\cite{JLPT}.

Szegedy \cite{Sz07} proposed a weakening of this conjecture, which is referred to as the Weak Additive Basis conjecture: For every prime $p\geq 3$ there exists a constant $c_2(p)$ such that if $B\subset \mathbb{F}_p^{n}$ is the union of $c_2(p)$ bases, then every $w\in\mathbb{F}_{p}^{n}$ can be written as $w=\sum_{v\in B}\alpha_v v$, where $\alpha_v\in \{1,\dots,p-1\}$ for every $v\in B$. We provide two results, both of which resolves this latter conjecture in a strong form for $p\geq 5$. In particular, one can take $c_2(p)=3$ if $p$ is sufficiently large.

\begin{theorem}\label{thm:main}
Let $p$ be a prime, and $n$ be a positive integer. 
\begin{itemize}
    \item[(i)] Let $p\geq 5$. There exists $A\subset \mathbb{F}_p$ of size $2\lfloor \log_2 p \rfloor$ such that the following holds. Let $B\subset \mathbb{F}_p^{n}$ be the union of $p$ bases, then every $w\in\mathbb{F}_p^{n}$ can be written as $w=\sum_{v\in B}\alpha_v v$, where $\alpha_v\in A$ for every $v\in B$.
    
    \item[(ii)] Let $p\geq 11$. Let $B\subset \mathbb{F}_p^{n}$ be the union of $3$ bases, then every $w\in\mathbb{F}_p^{n}$ can be written as a nonzero linear combination of elements of $B$.
\end{itemize}
\end{theorem}

Note that for $p\geq 5$ we have $2\lfloor \log_2 p\rfloor<p$. Also, if $A\subset \mathbb{F}_p$ satisfies the theorem, then every translate of $A$ also satisfies it, so Theorem \ref{thm:main} indeed implies the conjecture of Szegedy for every $p\geq 5$.

The Additive Basis conjecture is motivated by a celebrated conjecture of Tutte (see \cite{S76}) about the existence of nowhere zero 3-flows in graphs. A \emph{modulo-$k$-orientation} in a graph $G$ is an orientation of the edges such that for every vertex, the in- and outdegree are equal modulo $k$. Let $c_3(k)$ denote the smallest constant such that any $c_3(k)$-edge-connected graph has a modulo-$k$-orientation. The conjecture of Tutte (see \cite{S76}) states that $c_3(3)=4$, which was extended by Jaeger \cite{J84} to $c_3(k)=2k-2$ for every odd $k$. However, the statement that $c_3(3)$ exists is already not obvious, and it was a long-standing conjecture of Jaeger \cite{J88}. This was recently proved by  Thomassen \cite{T12} by showing that $c_3(3)\leq 8$. In a subsequent paper, the upper bound was further improved to $c_3(3)\leq 6$ by Lov\'asz et al. \cite{LTWZ}, who also established that $c_3(k)\leq 3k-3$ for every odd $k$. In \cite{JLPT}, it is demonstrated that if the Additive Basis conjecture is true, then $c_3(p)\leq 2c_1(p)$ (more precisely, this was shown for $p=3$, but the same argument works in general).

\subsection{Coset covers}

Say that a covering of a set with a collection of its subsets is \emph{irredundant}, if it contains no proper subcollection forming a covering. An old result of Neumann \cite{N54} states that if $G$ is a group and $\{H_ix_i:i\in [k]\}$ is an irredundant covering of $G$ with cosets, then $|G:\bigcap_{i\in [k]}H_i|$ is finite, and in \cite{N54+}, he proved that $|G:\bigcap_{i\in [k]}H_i|$ is bounded by a function of $k$. Therefore, it makes sense to define $f(k)$ denoting the maximum of $|G:\bigcap_{i\in [k]}H_i|$, where $G$ is a group and $\{H_ix_i:i\in [k]\}$ is an irredundant covering of $G$ with cosets. Similarly, define $g(k)$ to be the maximum of $|G:\bigcap_{i\in [k]}H_i|$ if $\{H_i:i\in[k]\}$ is an irredundant covering of $G$ with subgroups. Tomkinson \cite{T87} proved that $f(k)=k!$ for every $k$, the lower bound achieved by the symmetric group, while $\Omega(3^{2k/3})=g(k)<(k-2)^{3}(k-3)!$ for $k\geq 5$. It is a longstanding open problem whether $g$ is at most exponential. 

Pyber \cite{P96} further proposed to study the functions $f_{A}(k),g_{A}(k)$, which are defined as $f(k)$ and $g(k)$, respectively, with the additional condition that $G$ is abelian. Pyber noted that it would be already interesting to show that $g_{A}(k)$ is bounded by an exponential function, however, the best known bounds were still $g_{A}(k)\leq f_{A}(k)=e^{O(k\log k)}$. Szegedy \cite{Sz07} conjectured, that unlike in the non-abelian case, $f_{A}(k)=e^{O(k)}$ should also hold. We make the first improvement over these old bounds on $f_{A}(k)$ and $g_{A}(k)$, and show that $f_{A}(k)$ grows indeed slower than $f(k)$.

\begin{theorem}\label{thm:coset}
Let $A$ be an abelian group, and let $\{H_{i}x_i:i\in [k]\}$ be an irredundant coset cover of $A$. Then $$|A:\bigcap_{i\in [k]}H_{i}|=e^{O(k\log\log k)}.$$
\end{theorem}

A particularly interesting subcase of Theorem \ref{thm:coset} is when $A=\mathbb{F}_p^{n}$ is an \emph{elementary $p$-group} for some prime $p$. Szegedy \cite{Sz07} proposed the conjecture that if $p\geq 3$, then there exists $\epsilon(p)>0$ such that the codimension of $\bigcap_{i\in [k]}H_i$ is at most $k/(1+\epsilon(p))$. Szegedy showed that this conjecture is equivalent with the weak additive basis conjecture, and $\epsilon(p)\geq 1$ implies the Alon-Jaeger-Tarsi conjecture. One of our main lemmas, namely Lemma \ref{thm:hyperplane2}, states that this conjecture is indeed true for $p\geq 5$. More precisely, we show $\epsilon(p)=\Omega(\log p/\log\log p)$.

\subsection{Non-vanishing linear maps}

DeVos \cite{D00} proposed a substantial strengthening of the Alon-Jaeger-Tarsi conjecture, called the Choosability conjecture. A matrix $M\in \mathbb{F}_p^{n\times n}$ is \emph{$(a,b)$-choosable} if for all subsets $X_1,\dots,X_n, Y_1,\dots,Y_n\subset \mathbb{F}_p$ such that $|X_i|=a, |Y_i|=b$ for $i\in [n]$, there exists a vector $x\in X_1\times\dots\times X_n$ such that $Mx\in Y_1\times\dots\times Y_{n}$. The conjecture states that if $M$ is invertible, then it is $(k+2,p-k)$-choosable for every positive integer $k$. The first two authors \cite{NP21} proved that if $p\geq 61$, $p\neq 79$, then $M$ is $(p-1,p-1)$-choosable. Szegedy showed that this implies $\epsilon(p)\geq 1/2$.

Another far reaching generalization of the Alon-Jaeger-Tarsi conjecture was proposed by the first two authors \cite{NP21}. The following theorem resolves exactly this.

\begin{theorem}\label{thm:AT}
Let $k\geq 2$ be a positive integer, then there exists $p_0=p_0(k)$ such that the following holds for every positive integer $n$. Let $p>p_0$ be a prime, and let $M_1,\dots,M_{k}\in\mathbb{F}_{p}^{n\times n}$ be invertible matrices. Then there exists $x\in\mathbb{F}_p^{n}$ such that the vectors $M_1x,\dots,M_kx$ have no zero coordinates.
\end{theorem}

 Later, we further extend this theorem by providing a choosability version, which the interested reader can find as Theorem \ref{thm:strongAT}.

Our paper is organized as follows. In the next section, we introduce group rings and prove several lemmas about their properties. Many of these are reformulations and simplifications of results presented in \cite{NP21}. Then, in Section \ref{sect:additive}, we prove our main result about additive bases. In Section \ref{sect:coset}, we consider coset coverings of groups. Finally, in Section \ref{sect:maps}, we prove our result about non-vanishing linear maps.

\section{Group ring identities}\label{sect:groupring}

In this section, we establish several results about the group rings $\mathbb{F}_{p}[\mathbb{F}_p^{n}]$ and $\mathbb{C}[\mathbb{F}_p^{n}]$, which will serve as the main tools in our proofs. Given an additive group $G$ and a ring $R$, the \emph{group ring} $R[G]$ is the ring of formal expressions $\sum_{v\in G} r_v g^{v}$, where $r_v\in R$, and $g$ is a formal variable. Addition and multiplication are defined in the natural way, that is, $$\left(\sum_{v\in G} r_v g^{v}\right)+\left(\sum_{v\in G} r'_v g^{v}\right)=\sum_{v\in G} (r_v+r_v') g^{v},$$ and $$\left(\sum_{v\in G} r_v g^{v}\right)\cdot\left(\sum_{v\in G} r'_v g^{v}\right)=\sum_{v\in G} \left(\sum_{w\in G} r_w r'_{v-w}\right)g^{v}.$$
Note that an element $h=\sum_{v\in G} r_v g^{v}\in R[G]$ corresponds to the function $h^{*}:G\rightarrow R$ defined as $h^{*}(v)=r_v$. Then, the product $h_1\cdot h_2$ corresponds to the convolution $h_1^{*}*h_2^{*}$.

In this paper, we will consider the group rings $\mathbb{F}_{p}[\mathbb{F}_p^{n}]$ and $\mathbb{C}[\mathbb{F}_p^{n}]$, where $p$ is a prime. To simplify our notation, instead of $\sum_{y\in\mathbb{F}_p^{n}}$, we shall write $\sum_{y}$ if it causes no confusion. We will study the following notions.

\begin{definition}
For $r\in [p-1]$, say that a multiset $V\subset \mathbb{F}_p^{n}$ is \emph{$(r,\mathbb{F}_p)$-vanishing} if $$\prod_{v\in V}(1-g^{v})^{r}=0$$ 
in $\mathbb{F}_{p}[\mathbb{F}_p^{n}].$ Also, say that $V$ is   \emph{$(r,\mathbb{F}_p)$-irredundant} if $V$ is $(r,\mathbb{F}_p)$-vanishing, but no proper subset of $V$ is $(r,\mathbb{F}_p)$-vanishing. If $r=1$, we simply say $\mathbb{F}_p$-vanishing, and $\mathbb{F}_p$-irredundant, respectively.

Furthermore, $V$ is \emph{$(r,\mathbb{C})$-vanishing} if there exists $(t_v)_{v\in V}\in(\mathbb{F}_p)^{V}$ such that $$\prod_{v\in V}(1-\lambda^{t_{v}}g^{v})^{r}=0$$ 
in $\mathbb{C}[\mathbb{F}_p^{n}]$. Here, and in the rest of this paper, $\lambda=e^{2\pi i/p}$. Also, say that $V$ is  \emph{$(r,\mathbb{C})$-irredundant} if there exists $(t_v)_{v\in V}\in(\mathbb{F}_p)^{V}$ such that $\prod_{v\in V}(1-\lambda^{t_{v}}g^{v})^{r}=0$, but no proper subset $V'\subset V$ satisfies $\prod_{v\in V'}(1-\lambda^{t_{v}}g^{v})^{r}=0$. (Note that this is not equivalent to saying that no proper subset of $V$ is $(r,\mathbb{C})$-vanishing.) Again, if  $r=1$, we simply say $\mathbb{C}$-vanishing, and $\mathbb{C}$-irredundant, respectively.
\end{definition}

In \cite{NP21}, it was shown that if $V$ is $\mathbb{C}$-vanishing, then it is also $\mathbb{F}_p$-vanishing, however, we will not use this fact. The following simple observation is crucial:  being $\mathbb{C}$-vanishing and $\mathbb{F}_p$-vanishing is a projective property.

\begin{lemma}\label{lemma:minimal}
Let $R\in\{\mathbb{C},\mathbb{F}_{p}\}$, $r\in\mathbb{Z}^{+}$, let $V\subset\mathbb{F}_p^{n}$ be a multiset and $(a_{v})_{v\in V}\in(\mathbb{F}_p^{*})^{V}$. Set $V'=\{a_{v}v:v\in V\}$. Then $V$ is $(r,R)$-irredundant if and only if $V'$ is $(r,R)$-irredundant.
\end{lemma}

\begin{proof}
In case $R=\mathbb{F}_p$, this follows easily from the identity $$(1-g^{a_v v})=(1-g^{v})\cdot\left(\sum_{i=0}^{a_{v}-1}g^{iv}\right),$$ and in case $R=\mathbb{C}$, this follows from the identity $$(1-\lambda^{a_v t_v}g^{a_v v})=(1-\lambda^{t_v}g^{v})\cdot\left(\sum_{i=0}^{a_{v}-1}\lambda^{it_v}g^{iv}\right).$$
\end{proof}

One of the first (and very few) applications of group ring identities in combinatorics is the famous theorem of Olson \cite{O68} about vanishing sums in abelian groups whose order is a prime power. Olson's proof relied on an identity, which we state (and prove, for completeness) in a somewhat weaker form.

\begin{lemma}\label{lemma:olson}
Let $V\subset \mathbb{F}_p^{n}$ be a multiset of size at least $(p-1)n+1$. Then $V$ is $\mathbb{F}_p$-vanishing.
\end{lemma}
\begin{proof}
Let $e_1,\dots,e_n$ be a basis of $\mathbb{F}_p^{n}$. Note that if $v\in \mathbb{F}_p^{n}$, then we can write $$1-g^{v}=\sum_{i=1}^{n}f_{v,i}\cdot(1-g^{e_i})$$ with suitable $f_{v,1},\dots,f_{v,n}\in\mathbb{F}_p[\mathbb{F}_p^{n}]$. Indeed, if $v=\sum_{i=1}^{n}b_i e_i$, then
$$1-g^{v}=\sum_{i=1}^{n} g^{b_1e_1+\dots+b_{i-1}e_{i-1}}\cdot(1-g^{b_i e_i}),$$
so we can take $f_{v,i}=g^{b_1e_1+\dots+b_{i-1}e_{i-1}}\cdot(1+g^{e_i}+\dots+g^{(b_{i}-1)e_i})$. Consider the product
$$\prod_{v\in V}(1-g^{v})=\prod_{v\in V}\left(\sum_{i=1}^{n}f_{v,i}(1-g^{e_i})\right).$$
After expanding the outer brackets on the right hand side, we get a sum, whose every term has the form $f(1-g^{e_1})^{\alpha_1}\dots (1-g^{e_n})^{\alpha_n}$, where $\alpha_1+\dots+\alpha_n=|V|>(p-1)n$ and $f\in\mathbb{F}_p[\mathbb{F}_p^n]$. Therefore, in each such term, at least one of the $\alpha_i$'s is at least $p$. But $(1-g^{e_i})^{p}=0$, so every term evaluates to~0.
\end{proof}

\begin{corollary}\label{cor:olson}
If $r\in [p-1]$ and $V\subset \mathbb{F}_p^{n}$ is a multiset of size at least $((p-1)n+1)/r$, then $V$ is $(r,\mathbb{F}_p)$-vanishing.
\end{corollary}
\begin{proof}
 Apply Lemma \ref{lemma:olson} for the multiset $V'$ formed by $r$ copies of every element of $V$.
\end{proof}

Olson \cite{O68} observed that if in the product $\prod_{v\in V}(1-g^{v})=\sum_{y}c_y g^{y}$ the constant term vanishes (that is, $c_{0}=0$), then $V$ must contain a nonempty subset whose elements sum to 0. However, we will show that the whole product $\prod_{v\in V}(1-g^{v})$ being zero carries much more information about $V$. In particular, all of our main results rely on the following key lemma, which we state after providing a further definition. For integers $a<b$, let $[a,b]=\{a,a+1,\dots,b\}$.

\begin{definition}
 For $r\in [p-1]$, say that a set $A\subset \mathbb{F}_p$ is \emph{$r$-arithmetic} if for every $a\in A$ there exists $b\in\mathbb{F}_p^{*}=\mathbb{F}_p\setminus\{0\}$ such that $a+ib\in A$ for $i\in [-r,r]$, and for every $a\in \mathbb{F}_p\setminus A$, there exists $b\in A$ such that $a+ib\in A$ for $i\in [r]$. If $r=1$, say simply that $A$ is arithmetic.
\end{definition}

\noindent
In other words, a set $A$ is (1-)arithmetic if every element of $A$ is the middle term of a 3-term arithmetic progression. The study of arithmetic sets (also referred to as {\it balanced} sets) was initiated by Straus \cite{Straus76} in 1976, who showed that the minimum size of an arithmetic set is $\Theta(\log p)$. Browkin, Divi\v{s} and Schinzel \cite{BDS76} proved that the minimum is at least $1+\log_2 p$, while Nedev \cite{N09} established the almost matching upper bound $(1+o(1))\log_2 p$.

Now we are ready to state our key result.

\begin{lemma}\label{lemma:main}
Let $r\in[p-1]$ and $R\in \{\mathbb{F}_p,\mathbb{C}\}$, let $A\subset \mathbb{F}_p$ be an $r$-arithmetic set, and let $V\subset \mathbb{F}_p^{n}$ be an $(r,R)$-irredundant multiset. Then for every $x\in \langle V\rangle$ there exists $(a_v)_{v\in V}\in A^{V}$ such that
$$x=\sum_{v\in V}a_{v} v.$$
\end{lemma}

We prepare the proof of this lemma with another.

\begin{lemma}\label{lemma:choice}
Let $r\in[p-1]$, and $R\in \{\mathbb{F}_p,\mathbb{C}\}$, and let $V\subset \mathbb{F}_p^{n}$ be an $(r,R)$-irredundant multiset. Given $(b_v)_{v\in V}\in (\mathbb{F}_{p}^{*})^{V}$ and $w\in V$, there exists $(\epsilon_v)_{v\in V\setminus\{w\}}\in [-r,r]^{V\setminus \{w\}}$ and $\epsilon_w\in [r]$ such that 
$$\epsilon_w b_{w} w=\sum_{v\in V\setminus\{w\}} \epsilon_v b_v v.$$
\end{lemma}

\begin{proof}
Let us prove this for $R=\mathbb{F}_p$, the case $R=\mathbb{C}$ follows in the same manner. By Lemma \ref{lemma:minimal}, the set $V'=\{b_v v:v\in V\}$ is also $(r,\mathbb{F}_p)$-irredundant. Write
$$\prod_{v\in V\setminus\{w\}}(1-g^{b_{v}v})^{r}=\sum_{y}c_y g^{y},$$
where $c_{y}\in\mathbb{F}_p$. As $V'$ is $\mathbb{F}_p$-irredundant, there exists some $z\in\mathbb{F}_p^{n}$ such that $c_z\neq 0$. We make two observations.
\begin{itemize}
    \item[(i)] If $c_y\neq 0$, then $y$ is a linear combination of $(b_v v)_{v\in V\setminus \{w\}}$ with coefficients in $[0,r]$.
    \item[(ii)] If $c_y\neq 0$, then $c_{y-\epsilon b_w w}\neq 0$ for some $\epsilon\in [r]$. Consider the identity 
$$0=\prod_{v'\in V'}(1-g^{v'})^{r}=(1-g^{b_w w})^{r}\cdot\left(\sum_{y}c_y g^{y}\right)=\sum_{y} \left(\sum_{i=0}^{r}\binom{r}{i}(-1)^{i}c_{y-ib_w w}\right)g^{y}.$$
Therefore, $\sum_{i=0}^{r}\binom{r}{i}(-1)^{i}c_{y-ib_w w}=0$, but $c_y\neq 0$. Hence, $c_{y-b_w w}$,$c_{y-2b_w w}$,$\dots,c_{y-rb_w w}$ cannot be all zero.
\end{itemize}
Therefore, by (i) and (ii), there exist $\mu_{v},\mu'_{v}\in [0,r]$ for $v\in V\setminus \{w\}$ and $\epsilon_w\in [r]$ such that 
\begin{equation}\label{eq:z}
    z=\sum_{v\in V\setminus\{w\}} \mu_v b_v v
\end{equation}
 and
 \begin{equation}\label{equ:z-}
     z-\epsilon_w b_w w=\sum_{v\in V\setminus\{w\}} \mu'_v b_v v.
 \end{equation}
 Subtracting (2) from (1) gives $\epsilon_w b_w w=\sum_{v\in V\setminus \{w\}} (\mu_v-\mu'_v)b_v v$, so setting $\epsilon_v=\mu_v-\mu'_v$ for $v\in V\setminus\{w\}$ satisfies the desired properties.
\end{proof}

\begin{proof}[Proof of Lemma \ref{lemma:main}]
Let $x\in \langle V\rangle$, then $x$ can be written as a linear combination of elements of $V$. Among all such linear combinations, let $x=\sum_{v\in V}a'_{v}v$ be such that the number of $v\in V$ with $a_v'\not\in A$ is minimal. Our goal is to show that $a'_v\in A$ for every $v\in V$. Suppose this is not the case, and let $w\in V$ be such that $a_w'\not\in A$. Define the numbers $(b_v)_{v\in V}\in\mathbb{F}_p^{*}$ as follows. 
\begin{itemize}
    \item If $a'_{v}\in A$, then using the fact that $A$ is $r$-arithmetic, we can choose $b_{v}\in\mathbb{F}_p^{*}$ such that $a'_v+ib_v\in A$ for every $i\in [-r,r]$.
    \item If $a'_{v}\not\in A$ and $v\neq w$, choose $b_v$ arbitrarily. 
    \item Finally, choose $b_w$ such that $a'_w+ib_w\in A$ for every $i\in [r]$.
\end{itemize}
By Lemma \ref{lemma:choice}, there exists $(\epsilon_{v})_{v\in V}\in [-r,r]^{V\setminus \{w\}}$ and $\epsilon_w\in [r]$ such that $$\epsilon_w b_w w=\sum_{v\in V\setminus\{w\}} \epsilon_v b_{v} v.$$ Using this, we can write
$$x=(a'_w+\epsilon_w b_w)w+\sum_{v\in V\setminus\{w\}}(a'_v-\epsilon_v b_v)v.$$
Note that this contradicts the minimality of $(a'_v)_{v\in V}$. Indeed, setting $a''_w=a'_w+\epsilon_w b_w$ and $a''_v=a'_v-\epsilon_v b_v$ for $v\in V\setminus\{w\}$, the sequence $(a''_v)_{v\in V}$ contains strictly less elements not in $A$.
\end{proof}

As Lemma \ref{lemma:olson} tells us, if $V\subset \mathbb{F}_{p}^{n}$ is a multiset such that $|V|\geq (p-1)n+1$, then $V$ is $\mathbb{F}_p$-vanishing. Can we say something about the converse, that is, is there a lower bound on $|V|$ if $V$ is $\mathbb{F}_p$-vanishing? Note that if $V$ contains $p$ copies of the same vector, or more generally, if $V$ contains $(p-1)d+1$ vectors which span a space of dimension $d$, then $V$ is also $\mathbb{F}_p$-vanishing. We can show a weak converse of this, that is, if $V$ is $\mathbb{F}_p$-vanishing, then some subspace of $\mathbb{F}_{p}^{n}$ must contain too many elements of $V$.

\begin{lemma}\label{lemma:subspace}
Let $s$ be the size of the smallest arithmetic set of $\mathbb{F}_p$, let $R\in\{\mathbb{C},\mathbb{F}_p\}$, and let $V\subset \mathbb{F}_p^{n}$ be an $R$-irredundant multiset. Then 
$$|V|\geq \frac{\log p}{\log s} \dim\langle V\rangle.$$
\end{lemma}

\begin{proof}
Let $A\subset \mathbb{F}_{p}$ be an arithmetic set of size $s$, and let $d=\dim\langle V\rangle$. By Lemma \ref{lemma:main}, every $x\in \langle V\rangle$ can be written as $x=\sum_{v\in V}\alpha_v v$ with suitable $\alpha_v\in A$ for $v\in V$. But note that $\langle V\rangle$ contains $p^{d}$ elements, while the number of linear combinations $\sum_{v\in V}\alpha_v v$ satisfying $\alpha_v\in A$ is $s^{|V|}$. Therefore, we get $s^{|V|}\geq p^{d}$, finishing the proof.
\end{proof}

Finally, we have the following results about the sizes of arithmetic sets. While the aforementioned result of Nedev \cite{N09} implies the existence of arithmetic sets of size $(1+o(1))\log_2 p$, we are also interested in small values of $p$, so we will rely on an exact result instead. Indeed, it is crucial that for $p\geq 5$, $\mathbb{F}_p$ contains arithmetic sets of size strictly less than $p$.

\begin{lemma}\label{lemma:arithmetic}
Let $p\geq 5$ be a prime.
\begin{itemize}
    \item[(i)] \emph{(\cite{NP21})} There exists an arithmetic set $A\subset\mathbb{F}_p$ of size $2\lfloor \log_2 p\rfloor$.
    
    \item[(ii)] $A=\mathbb{F}_p^{*}$ is $(p-3)/2$-arithmetic.
\end{itemize}
\end{lemma}

\begin{proof}
 We only prove (ii). Let $a\in\mathbb{F}_p$, then it is enough to show that there exists $b\in \mathbb{F}^*_p$ such that $a+ib\neq 0$ if $i\in [-(p-3)/2,(p-3)/2]$, $i\neq 0$. If $a=0$, take $b=1$, and if $a\neq 0$, let $b=2a$. These choices suffice.
\end{proof}

\section{Additive bases}\label{sect:additive}

Let us start with the proof of Theorem \ref{thm:main}. In particular, we prove the following slightly more general result, which then almost immediately implies it.

\begin{theorem}\label{thm:AB}
Let $p$ be a prime, $r\in [p-1]$, and let $A\subset \mathbb{F}_p$ be an $r$-arithmetic set. If $B\subset \mathbb{F}_p^{n}$ is the union of at least $p/r$ bases, then every $w\in\mathbb{F}_p^{n}$ can be written as $w=\sum_{v\in B}\alpha_v v$, where $\alpha_v\in A$ for every $v\in B$.
\end{theorem}

\begin{proof}
We proceed by induction on $n$. In the base case $n=0$ there is nothing to prove, so suppose that $n\geq 1$. Let $B\subset\mathbb{F}_p^{n}$ be such that $B$ is the union of at least $p/r$ bases, then $|B|\geq pn/r$. Hence, by Corollary \ref{cor:olson}, $B$ is $(r,\mathbb{F}_p)$-vanishing. But then $B$ contains an $(r,\mathbb{F}_p)$-irredundant subset $V$.

Let $T=\langle V\rangle$, and $S=\mathbb{F}_p^{n}/T\cong \mathbb{F}_p^{n-\dim(T)}$. Then every $x\in \mathbb{F}_p^{n}$ can be written as $x=(x_S,x_T)$ with $x_S\in S$ and $x_T\in T$. Let $B'=\{b_S:b\in B\setminus V\}\subset S$. Note that if $C$ is a basis in $B$, then $C'=\{b_S:b\in C\setminus V\}$ contains a basis in $B'$, so $B'$ contains the union of $p/r$ bases of $S$. Therefore, by our induction hypothesis applied to $S$, for every $x\in\mathbb{F}_p^n$, we can write 
$$x_S=\sum_{v\in B\setminus V}\alpha_{v}v_S$$
with suitable $\alpha_v\in A$, $v\in B\setminus V$. Also, by Lemma \ref{lemma:main} applied to $x'=x_T-\sum_{v\in B\setminus V}\alpha_{v}v_T\in T$, there exists $\alpha_v\in A$ for every $v\in V$ such that 
$$x_T-\sum_{v\in B\setminus V}\alpha_{v}v_T=\sum_{v\in V}\alpha_v v_T.$$
But then $x=\sum_{v\in B}\alpha_v v$, finishing the proof.
\end{proof}

\begin{proof}[Proof of Theorem \ref{thm:main}]
 In order to prove (i), apply Theorem \ref{thm:AB} with $r=1$ and an arithmetic set $A\subset \mathbb{F}_p$ of size $2\lfloor \log_2 p\rfloor$, whose existence is guaranteed by Lemma \ref{lemma:arithmetic}. To prove (ii), apply Theorem~\ref{thm:AB} with $r=(p-3)/2$ and $A=\mathbb{F}_p^{*}$, which is $(p-3)/2$-arithmetic by Lemma \ref{lemma:arithmetic}. Note that $p/r< 3$ if $p\geq 11$.
\end{proof}

\section{Coset covers}\label{sect:coset}

In this section, we prove Theorem \ref{thm:coset}. For a group $G$, let $\phi(G)$ denote the smallest $k$ for which there exists an irredundant coset cover $\{H_ix_i:i\in [k]\}$ such that $\bigcap_{i\in [k]} H_{i}$ is trivial. Note that Theorem \ref{thm:coset} is equivalent with the statement that for every finite abelian group $A$, we have $\phi(A)=\Omega(\log |A|/\log\log\log|A|)$. In particular, we prove a slightly stronger result.

\begin{theorem}\label{thm:coset2}
There exists an absolute constant $c>0$ such that the following holds. Let $A$ be a finite abelian group and let $p_1^{n_1}\dots p_m^{n_m}$ be the prime factorization of $|A|$. Then 
$$\phi(A)\geq \sum_{i\in [m]} cn_{i}\frac{\log p_i}{\log\log (p_i+1)}.$$
\end{theorem}

We prepare the proof of this theorem with several statements. First, we translate Lemma \ref{lemma:subspace} into a statement about hyperplane covers of the elementary $p$-group $\mathbb{F}_p^{n}$. With slight abuse of notation, we consider $\mathbb{F}_p^{n}$ as an additive group.

\begin{lemma}\label{thm:hyperplane2}
Let $s$ be the size of the smallest arithmetic set of $\mathbb{F}_p$, and let $H_{1},\dots,H_{k}\subset \mathbb{F}_p^{n}$ be linear hyperplanes such that some translates $H_1+z_1,\dots,H_{k}+z_k$ form an irredundant cover of $\mathbb{F}_p^{n}$. Then the codimension of $\bigcap_{i=1}^{k}H_{i}$ is at most $k\log s/\log p$.
\end{lemma}

\begin{proof} For $i=1,\dots,k$, let $v_i\in \mathbb{F}_p^{n}$  such that $H_{i}=\{x\in \mathbb{F}_p^{n}:\langle x,v_i\rangle=0\}$, and let $V=\{v_i:i\in [k]\}$. We show that $V$ is $\mathbb{C}$-irredundant.  Let us recall that if $h=\sum_{v}c_v g^{v}\in \mathbb{C}[\mathbb{F}_p^{n}]$, then $h^{*}:\mathbb{F}_p^{n}\rightarrow \mathbb{C}$ denotes the function defined as $h^{*}(v)=c_v$, and for $h_1,h_2\in\mathbb{C}[\mathbb{F}_p^{n}]$, we have $(h_1\cdot h_2)^{*}=h_1^{*}*h_2^{*}$. Given a function $f:\mathbb{F}_p^{n}\rightarrow \mathbb{C}$, denote its discrete Fourier transform by $F(f)$. That is, $F(f)(x)=\sum_{v}\lambda^{\langle x,v\rangle}f(v)$. The Fourier transform turns convolution into product, so for $h_1,h_2\in \mathbb{C}[\mathbb{F}_p^{n}]$, we have 
$$F((h_1\cdot h_2)^{*})=F(h_1^*)\cdot F(h_2^{*}).$$

For  $t\in\mathbb{F}_p$ and $v\in\mathbb{F}_p^{n}$, let $h_{v,t}=1-\lambda^{t}g^{v}$. Then $F(h^{*}_{v,t})(x)=1-\lambda^{t+\langle x,v\rangle}$. Therefore, the set of points the function $F(h_{v,t}^{*})$ vanishes at is the hyperplane $\{x:\langle x,v\rangle=-t\}$. In general, given $t_1,\dots,t_k\in\mathbb{F}_p$, the set of points the function $(h_{v_1,t_1}\dots h_{v_k,t_k})^*$ vanishes at is the union of the hyperplanes $H_{i}'=\{x\in\mathbb{F}_p^{n}:\langle x,v_i\rangle=-t_i\}$. This shows that $\prod_{i=1}^{k}(1-\lambda^{t_i}g^{v_i})=0$ if and only if $H_1',\dots,H_{k}'$ form a covering of $\mathbb{F}_p^{n}$. From this, we deduce that some translates of $H_1,\dots,H_k$ form an irredundant cover if and only if $V$ is $\mathbb{C}$-irredundant. We finish the proof by citing Lemma \ref{lemma:subspace}: $$\frac{\log s}{\log p}|V|\geq  \dim\langle V\rangle= \mbox{codim}\left(\bigcap_{i=1}^{k}H_{i}\right).$$
\end{proof}

As discussed in the introduction, the previous lemma, combined with Lemma \ref{lemma:arithmetic}, resolves Conjecture 10 of Szegedy \cite{Sz07} for $p\geq 5$ in a strong sense. Note that Lemma \ref{thm:hyperplane2} also confirms Theorem \ref{thm:coset2} in case $A=\mathbb{F}_p^{n}$, assuming we are only allowed maximal cosets. In the rest of this section, our goal is to reduce the general case to this special setup. 

\begin{definition}
 A coset cover $\{H_ix_i:i\in[k]\}$ of an abelian group $A$ is \emph{efficient}, if it is irredundant, $\bigcap_{i\in [k]}H_{i}$ is trivial, and $H_{i}$ is a maximal subgroup of $A$ for $i\in [k]$. 
\end{definition}

\begin{lemma}\label{lemma:fettuccine}
If $A$ has an efficient coset cover, then $A\cong \mathbb{F}_p^{n}$ for some prime $p$ and $n\in\mathbb{Z}^{+}$.
\end{lemma}

\begin{proof}
Let $\{H_ix_i:i\in[k]\}$ be an efficient coset cover of $A$. By the fundamental theorem of finite abelian groups, we can write $A=A_1\oplus\dots\oplus A_{m}$, where $|A_{1}|,\dots,|A_{m}|$ are  powers of distinct primes. First, we show that $m=1$. Suppose that $m>1$, and for $i\in [m]$, $a\in A$, let  $\pi_i(a)$ denote the projection of $a$ into $A_{i}$. As $H_{i}$ is a maximal subgroup of $A$, there is a unique $\tau(i)\in [m]$ such that $\pi_{\tau(i)}(H_{i})$ is a maximal subgroup of $A_{\tau(i)}$, and $\pi_{j}(H_{i})=A_{j}$ for $j\in [m]\setminus\{\tau(i)\}$. For $j\in [m]$, let $J_{j}\subset [k]$ be the set of indices $i$ such that $\tau(i)=j$. Note that $J_{j}$ is nonempty for every $j\in [m]$, otherwise $A_{j}<\pi_{j}(\bigcap_{i\in [k]}H_{i})$. Furthermore, $\{H_{i}x_i:i\in J_{j}\}$ does not cover at least one element $a\in A$, so it does not cover any element $b\in A$ with $\pi_{j}(b)=\pi_{j}(a)=:\alpha_j$. But then $\{H_{i}x_i:i\in [k]\}$ does not cover $(\alpha_1,\dots,\alpha_m)$, contradiction.

Now we can assume that $|A|$ is a power of some prime $p$. The intersection of all maximal subgroups of $A$, denoted by $\mbox{Fr}(A)$, is called the Frattini subgroup \cite{F85}. It is known that if $A$ is a $p$-group, then $\mbox{Fr}(A)$ is the smallest normal subgroup $N$ such that $A/N\cong\mathbb{F}_p^{n}$ for some $n$. Note that if $A$ has an efficient coset cover, then $\mbox{Fr}(A)$ is trivial, therefore, $A\cong\mathbb{F}_p^{n}$ for some $n\in\mathbb{Z}^{+}$.
\end{proof}

For a prime $p$ and $n\in\mathbb{Z}^{+}$, define $\phi(p,n)$ to be the size of a minimal sized efficient coset cover of $\mathbb{F}_p^{n}$. Also, define $\lambda_p:=\inf_{n\in \mathbb{Z}^{+}}\frac{\phi(p,n)-1}{n}$. 

\begin{lemma}\label{lemma:gpn}
Let $s$ be the size of the smallest arithmetic set of $\mathbb{F}_p$. Then $\phi(p,n)\geq \frac{\log p}{\log s}n$.
\end{lemma}

\begin{proof}
Let $k=\phi(p,n)$. Then there exist $k$ linear hyperplanes $H_1,\dots,H_k\subset \mathbb{F}_p^{n}$ and $k$ vectors $z_1,\dots,z_k\in\mathbb{F}_p^{n}$ such that  $\{H_i+z_i:i\in [k]\}$ is an efficient coset cover of $\mathbb{F}_p^{n}$. This means that $\{H_i+z_i:i\in [k]\}$ is irredundant, and $\bigcap_{i=1}^{k}H_{i}$ is trivial. In other words, the codimension of $\bigcap_{i=1}^{k}H_{i}$ is $n$, which implies $k\log s/\log p\geq n$ by Lemma \ref{thm:hyperplane2}. This finishes the proof.
\end{proof}

\begin{corollary}\label{cor:lambda}
There exists a constant $c>0$ such that for every prime $p$, we have $\lambda_p\geq c\frac{\log p}{\log\log (p+1)}.$
\end{corollary}

\begin{proof}
Let $s$ be the size of the smallest arithmetic set of $\mathbb{F}_p$, then $s\leq 2\lfloor \log_2 p\rfloor$ if $p\neq 3$, and $s=3$ if $p=3$, so $s\leq 2\log_2 p$ holds in general. Also, we have $\phi(p,n)\geq n\log p/\log s$ by Lemma \ref{lemma:gpn}. But note that $\phi(p,n)\geq 2$ also holds, so $\phi(p,n)-1\geq \phi(p,n)/2$. Therefore, $\frac{\phi(p,n)-1}{n}\geq \frac{\log p}{2\log (2 \log_2 p)}\geq  \frac{c\log p}{\log\log (p+1)}$ for some absolute constant $c>0$.
\end{proof}

For every $N\in\mathbb{Z}^{+}$ with prime factorization $N=p_1^{n_1}\dots p_{m}^{n_m}$, define $\lambda(N)=\sum_{i=1}^{m}n_{i}\lambda_{p_i}$. Note that $\lambda$ satisfies $\lambda(a\cdot b)=\lambda(a)+\lambda(b)$. By Corollary \ref{cor:lambda}, we have $\lambda(N)=\Omega(\sum_{i=1}^{m}n_{i}\frac{\log p_i}{\log\log (p_i+1)})$, so Theorem \ref{thm:coset2} is an immediate consequence of the following lemma.

\begin{lemma}\label{lemma:coset}
Let $A$ be a finite abelian group. Then $\phi(A)\geq \lambda(|A|)+1$.
\end{lemma}

The proof of this lemma follows closely an argument of Szegedy \cite{Sz07}. We will use the following simple claim repeatedly.

\begin{claim}\label{claim:subcover}
Let $\{H_ix_i:i\in[k]\}$ be an irredundant coset cover of the group $A$. Then for every $j\in [k]$, we have $\bigcap_{i\in [k]}H_{i}=\bigcap_{i\in [k]\setminus\{j\}} H_{i}$.
\end{claim}

\begin{proof}
Let $X=A\setminus (\bigcup_{i\in [k]\setminus\{j\}} H_{i}x_i)$. Then $X$ is nonempty, as $\{H_{i}x_i:i\in [k]\}$ is irredundant. But then $X$ is the union of cosets of $\bigcap_{i\in [k]\setminus\{j\}}H_{i}$. As $X\subset H_{j}x_{j}$, we must have $\bigcap_{i\in [k]\setminus\{j\}}H_{i}\subset H_{j}$, finishing the proof.
\end{proof}

\begin{proof}[Proof of Lemma \ref{lemma:coset}]
We proceed by induction on $|A|$. In case $|A|=1$, the statement is trivial, so suppose that $|A|\geq2$. Let $k=\phi(A)$, and let $\mathcal{C}=\{H_ix_{i}:i\in [k]\}$ be an irredundant coset cover of $A$ such that $\bigcap_{i\in [k]}H_{i}$ is trivial.

Let $M$ be the number of non-maximal subgroups among $H_1,\dots,H_{k}$. We will also proceed by induction on $M$. In case $M=0$, the coset covering $\{H_ix_i:i\in [k]\}$ is also efficient, so $A\cong\mathbb{F}_p^{n}$ for some prime $p$ and $n\in\mathbb{Z}^{+}$ by Lemma \ref{lemma:fettuccine}. Hence, $k\geq \phi(p,n)\geq \lambda_{p} n+1=\lambda(|A|)+1$, and we are done. 

Therefore, we can assume that $M\geq 1$, and without loss of generality, $H_k$ is not a maximal subgroup of $A$. Replace $H_{k}$ with some maximal subgroup $H_k'<A$ containing $H_k$. Let $\mathcal{C}'=\{H_i x_i:i\in [k-1]\}\cup\{H_{k}'x_{k}\}$, then $\mathcal{C}'$ is a coset covering, and $H_{k}'\cap \bigcap_{i\in [k-1]}H_{i}$ is trivial by Claim \ref{claim:subcover}. Note that there are $M-1$ non-maximal subgroups among $H_1,\dots,H_{k-1},H_{k}'$, so if $\mathcal{C}'$ is irredundant, we are done by our induction hypothesis.

Therefore, we can assume that $\mathcal{C}'$ is not irredundant, so, without loss of generality, there exists $\ell\leq k-2$ such that $\mathcal{C}''=\{H_{i}x_i:i\in [\ell]\}\cup\{H_{k}'x_{k}\}$ is an irredundant cover of $A$. As this cover contains less than $k=\phi(A)$ cosets, we must have that $B=H'_{k}\cap\bigcap_{i\in [\ell]}H_{i}$ is nontrivial. Therefore, using our first induction hypothesis, we get
$$\ell+1\geq\phi(A/B)\geq \lambda(|A/B|)+1=\lambda(|A|)-\lambda(|B|)+1.$$

For $t=0,1,\dots$, we define the sequence of 4-tuples $(B_t,X_t,Y_t,I_t)$, where $B_t<A$ is a subgroup, $X_t,Y_t\subset A$ are subsets and $I_t\subset [k]\setminus[\ell]$ is an index set, in
such a way that the following properties hold.
\begin{enumerate}
    \item[(i)] $X_t=\bigcup_{y\in Y_t}B_ty$,
    \item[(ii)] $\{H_{i}x_i:i\in I_t\}$ is an irredundant cover of $X_t$,
	\item[(iii)]  $B_t\cap \bigcap_{i\in I_t}H_{i}$ is trivial,
	\item[(iv)] $k-|I_t|\geq \lambda(|A|)-\lambda(|B_t|)+t$.
\end{enumerate}
Set $B_0:=B$, $X_0:=A\setminus\bigcup_{i\in [\ell]}H_ix_i$, and $I_0:=[k]\setminus[\ell]$. Note that $B_0=\bigcap_{i\in [\ell]}H_i$ holds by Claim \ref{claim:subcover}, so $X_0$ is a union of cosets of $B_0$. Therefore, there exists $Y_{0}\subset A$ such that $X_0=\bigcup_{y\in Y_0} B_0y$.  Also, we have $k-|I_0|=\ell\geq \lambda(|A|)-\lambda(|B_0|)$, so the choice $(B_0,X_0,Y_0,I_0)$ satisfies (i)-(iv). If $B_{t},X_t,Y_{t},I_t$ are already defined satisfying the above properties, we proceed as follows. 
Suppose that $B_{t}$ is non-trivial, then $I_{t},X_t,Y_t$ are nonempty. (Note that for the initial step $t=0$ these indeed hold.) By (iii), there exists $j\in I_{t}$ such that $H_j$ does not contain $B_{t}$, and by (ii), there exists some $x\in X_t$ which is only covered by $H_{j}x_{j}$. Let $y\in Y_{t}$ be such that $x\in B_ty$, and let $J\subset I_t$ be a set of indices such that $\{H_ix_i:i\in J\}$ is an irredundant cover of $B_ty$. Set $B_{t+1}:=B_t\cap \bigcap_{i\in J}H_{i}$, then $|J|\geq \phi(B_t/B_{t+1})\geq\lambda(|B_t|)-\lambda(|B_{t+1}|)+1$ by our induction hypothesis. Set $I_{t+1}:=I_t\setminus J$. Observe that we have $k-|I_{t+1}|=k-|I_t|+|J|\geq \lambda(|A|)-\lambda(|B_{t+1}|)+1$.

If $B_{t+1}$ is a trivial subgroup, then we stop. Note that in this case (iv) implies $k\geq k-|I_{t+1}|\geq \lambda(|A|)+1$, finishing our proof.

If $B_{t+1}$ is non-trivial, then $I_{t+1}$ and $X_{t+1}:=X_{t}\setminus (\bigcup_{i\in J}H_ix_i)$ are nonempty (by using (iii) and (ii), respectively). Also, $X_{t+1}$ is the union of cosets of $B_{t+1}$, so there exists $Y_{t+1}\subset A$ such that $X_{t+1}=\bigcup_{y\in Y_{t+1}}B_{t+1}y$.  Hence, (i)-(iv) are satisfied for $(B_{t+1},X_{t+1},Y_{t+1},I_{t+1})$ as well. Note that $B_{t+1}$ is a proper subgroup of $B_{t}$, so the sequence stops after a finite number of steps, giving the desired result.
\end{proof}

This finishes the proof of Theorem~\ref{thm:coset2}. Note that this theorem immediately implies that if $p$ is the largest prime divisor of $|A|$, then $\phi(A)\geq c\frac{\log |A|}{\log\log (p+1)}$. In case $|A|=e^{\Omega(p)}$, this gives our desired bound $\phi(A)\geq \Omega(\frac{\log |A|}{\log\log \log |A|})$. However, in case $|A|=e^{O(p)}$, there is an even simpler argument.

\begin{lemma}\label{lemma:simple}
Let $A$ be a finite abelian group, and let $p$ be a prime divisor of $|A|$. Then $\phi(A)\geq p$.
\end{lemma}

\begin{proof}
Let $k=\phi(A)$, and let $\{H_ix_i:i\in [k]\}$ be an irredundant coset covering of $A$ such that $\bigcap_{i\in [k]}H_{i}$ is trivial. Furthermore, let $B<A$ be the unique maximal $p$-subgroup of $A$. Without loss of generality $B\not<H_1$. There exists some $a\in A$ which is only covered by $H_1x_1$, let $By$ be the coset of $B$ containing $a$. Then for any $i\in [k]$, we have $|H_{i}x_{i}\cap By|\leq |B|/p$, so $\{H_ix_i:i\in [k]\}$ must contain at least $p$ cosets in order to cover $By$.
\end{proof}

\begin{proof}[Proof of Theorem \ref{thm:coset}]
Let $A'=A/\bigcap_{i\in [k]}H_i$, let $N=|A'|$, and let $p$ be the largest prime divisor of $N$. If $N<e^{p}$, then using Lemma~\ref{lemma:simple}, we get $k\geq \phi(A')\geq p$, which gives $N<e^{k}$. On the other hand, if $N\geq e^{p}$, we can use Theorem~\ref{thm:coset2} to conclude that there exists a constant $c>0$ such that  $k\geq \phi(A')\geq c\log N/\log\log (p+1)\geq c\log N/\log\log\log (N+1)$. From this, we get the desired result $N\leq e^{c'k\log\log k}$, where $c'>0$ is some absolute constant.
\end{proof}

\section{Non-vanishing linear maps}\label{sect:maps}

Finally, let us prove Theorem \ref{thm:AT}, which turns out to be a simple consequence of Lemma \ref{thm:hyperplane2}. In particular, we prove a stronger choosability version of the theorem, as promised in the introduction.

\begin{theorem}\label{thm:strongAT}
Let $k,r$ be positive integers, let $p$ be a prime, and let $s$ be the size of the smallest arithmetic set in $\mathbb{F}_p$. Suppose that $s^{kr}<p$, and let $M_1,\dots,M_{k}\in\mathbb{F}_{p}^{n\times n}$ be invertible matrices. Then given $k\cdot n$ sets $X_{i,j}\subset \mathbb{F}_p$ of size $p-r$ for $(i,j)\in [k]\times [n]$, there exists $x\in\mathbb{F}_p^{n}$ such that $(M_ix)(j)\in X_{i,j}$ for $(i,j)\in [k]\times [n]$.
\end{theorem}

\begin{proof}
 Let $v_{i,j}$ denote the $j$'th row of $M_i$, and for $t\in\mathbb{F}_p\setminus X_{i,j}$, let $v_{i,j,t}$ be a copy of $v_{i,j}$. Let $J_0=\{(i,j,t)\in [k]\times [n]\times\mathbb{F}_p:t\in X_{i,j}\}$ be the set of available indices. Note that $(M_{i}x)(j)=\langle v_{i,j,t},x\rangle$. Suppose the theorem does not hold, that is, for every $x$, there exists $(i,j)\in [k]\times [n]$ such that $(M_{i}x)(j)\not\in X_{i,j}$. For every $(i,j,t)\in J_0$, define the linear hyperplane $H_{i,j,t}=\{x\in\mathbb{F}_{p}^{n}: \langle x,v_{i,j}\rangle=0\}$, and its translation $H_{i,j,t}'=\{x\in\mathbb{F}_{p}^{n}: \langle x,v_{i,j}\rangle=t\}$. Then the system of affine hyperplanes $H_{i,j,t}'$ form a covering of $\mathbb{F}_p^{n}$. Therefore, one can select a subsystem which forms an irredundant covering of $\mathbb{F}_p^{n}$, let $J\subset J_0$ be the corresponding set of indices. By Lemma \ref{thm:hyperplane2}, we deduce that 
 \begin{equation}\label{eq1}
      \frac{\log s}{\log p}|J|\geq \mbox{codim}\left(\bigcap_{(i,j,t)\in J}H_{i,j,t}\right)=\dim\langle v_{i,j,t}:(i,j,t)\in J\rangle.
 \end{equation}
 However, note that for every $I\subset J_0$, we have 
 \begin{equation}\label{eq2}
     \dim\langle v_{i,j,t}:(i,j,t)\in I\rangle\geq \frac{|I|}{kr}.
 \end{equation}
 Indeed, by the pigeonhole principle, there exists $i_{0}\in [k]$ and $I'\subset I$ such that $|I'|\geq |I|/k$, and every $(i,j,t)\in I'$ satisfies $i=i_0$. Also, each copy of the vector $v_{i_0,j}$ can appear at most $r$ times in $I'$, so $I'$ contains at least $|I'|/r\geq |I|/kr$ different rows of $M_{i_0}$. As $M_{i_0}$ is invertible, these rows span a space of dimension at least $|I|/kr.$ Comparing (\ref{eq1}) and (\ref{eq2}), we get 
 $$\frac{\log s}{\log p}|J|\geq \frac{|J|}{kr},$$ contradicting the condition $s^{kr}<p$.
\end{proof}

\begin{proof}[Proof of Theorem \ref{thm:AT}]
Apply Theorem \ref{thm:strongAT} with parameters $r=1$, and $X_{i,j}=\mathbb{F}_p^{*}$ for $(i,j)\in [k]\times [n]$. Note that $s\leq 2\log_2 p$ by Lemma \ref{lemma:arithmetic}, so setting $p_0(k)=k^{10k}$, we have $s^{kr}<p$ satisfied for $p>p_0$. This finishes the proof.
\end{proof}

\section*{Acknowledgments}
J.~N. and P.~P.~P. was supported by the Lend\"ulet program of the Hungarian Academy of Sciences (MTA). P.~P.~P. was also partially supported by the National Research, Development and Innovation Office NKFIH (Grant Nr. K124171 and K129335).

\noindent
I.~T. was supported by the SNSF grant 200021\_196965, and also acknowledges the support of Russian Government in the framework of MegaGrant no 075-15-2019-1926, and the support of MIPT Moscow.

\end{document}